\documentclass[10pt,a4paper]{amsart}

\usepackage{amssymb,amsmath,amsfonts}

\usepackage{graphicx}

\usepackage{mathpazo}
\usepackage[hyperindex,pageanchor]{hyperref}

\numberwithin{equation}{section}
\newtheorem{theorem}{Theorem}[section]
\newtheorem{lemma}[theorem]{Lemma}
\newtheorem{proposition}[theorem]{Proposition}
\newtheorem{corollary}[theorem]{Corollary}

\theoremstyle{definition}

\theoremstyle{remark}
\newtheorem{remark}[theorem]{Remark}
\newtheorem{example}[theorem]{Example}
\newtheorem{question}[theorem]{Question}

\newcommand{\fgrade}{\operatorname{fgrade}}

\newcommand{\Ht}{\operatorname{ht}}

\newcommand{\Hom}{\operatorname{Hom}}

\newcommand{\Rad}{\operatorname{Rad}}

\newcommand{\depth}{\operatorname{depth}}

\newcommand{\cd}{\operatorname{cd}}

\newcommand{\vpl}{\operatornamewithlimits{\varprojlim}}

\newcommand{\fm}{\frak{m}}

\newcommand{\fp}{\frak{p}}

\newcommand{\fa}{\frak{a}}
\newcommand{\fb}{\frak{b}}
\newcommand{\fc}{\frak{c}}
\newcommand{\fx}{\underline{x}}
\newcommand{\fn}{\frak{n}}

\begin{document}

\author[Eghbali]{M. Eghbali \   }
\author[Shirmohammadi]{\   N. Shirmohammadi }

\title[On cohomological dimension and depth under linkage]
{On cohomological dimension and depth under linkage}

\address{School of Mathematics, Institute for Research in Fundamental Sciences (IPM), P. O. Box: 19395-5746, Tehran-Iran}
\email{m.eghbali@yahoo.com}

\address{Department of Mathematics, University of Tabriz, Tabriz, Iran. }
\email{shirmohammadi@tabrizu.ac.ir}

\thanks{The first author is supported in part by a grant from IPM (No. 91130026)}.



\begin{abstract}
Some relations between cohomological dimensions and depths of linked
ideals are investigated and discussed by various examples.
\end{abstract}

\maketitle

\section{Introduction}

Throughout this paper, we assume that $R$ is a commutative
Noetherian ring and $\fa$ is an ideal of $R$. Let $M$ be an
$R$-module. For an integer $i \in \mathbb {Z}$, let $H^i_{\fa}(M)$
denote the $i$th local cohomology module of $M$ with respect to
$\fa$ as introduced by Grothendieck (cf. \cite{Gr} and
\cite{Br-Sh}).

One of the most interested subjects in commutative algebra and
related topics is the cohomological dimension of an ideal. The
cohomological dimension of $\fa$ in $R$, $\cd(\fa,R)$ is defined as
 $$\cd(\fa,R)=\min\{i :  H^j_{\fa}(R) = 0 \text{\ for all\ } j > i\}.$$
In the light of Grothendieck's Vanishing Theorem, $\cd(\fa,R) \leq
\dim R$ and the equality happens when $(R,\fm)$ is a local ring and
$\fa$ is an $\fm$-primary ideal. The Hartshorne-Lichtenbaum
Vanishing Theorem provides conditions for $\cd(\fa,R) \leq \dim
R-1$, where $\fa$ is not an $\fm$-primary ideal. The cohomological
dimension has been studied by several authors; see for instance
\cite{F}, \cite{Ha}, \cite{O}, \cite{Ha-Sp}, \cite{Hu-Lyu},
\cite{D-N-T} and \cite{Sch2}.

Let $(R,\fm)$ be a Gorenstein local ring. Two ideals $\fa$ and $\fb$
of $R$ are linked by a Gorenstein ideal $\fc$ if $\fa=\fc:\fb$ and
$\fb=\fc:\fa$. Recall that an ideal $\fc$ is a Gorenstein ideal if
$R/\fc$ is a Gorenstein ring. From now on we mean $\fa$ and $\fb$
are linked when they are linked by a Gorenstein ideal. Iterating the
procedure, we say that an ideal $\fb$ is in the linkage class of an
ideal $\fa$ if there are ideals $\fc_1,\ldots,\fc_n$ such that $\fb$
is linked to $\fc_1$, $\fc_1$ is linked to $\fc_2$, ..., and $\fc_n$
is linked to $\fa$. If, in addition, $n$ is odd, we say that $\fb$
is in the even linkage class of $\fa$, that is $\fb$ is evenly
linked to $\fa$. In case $\fa$ is an unmixed ideal and
$\fc\subset\fa$ is a Gorenstein ideal with $\Ht\fa=\Ht\fc$, then
$\fa$ and $\fc:\fa$ are linked by $\fc$ (cf. \cite[Proposition
2.2]{Sch}).

Below, we recall some of the known results about the linkage that
will be used in the course of the paper. We refer the reader to
\cite{Sch} or \cite{St-V}.

\begin{lemma}\label{Schenzel}
Suppose that $(R,\fm)$ is a Gorenstein local ring, and $\fa$ and
$\fb$ are two linked ideals of $R$ by a Gorenstein ideal $\fc$.
\begin{enumerate}
  \item[(a)] There exists a short exact sequence
  $0\longrightarrow K_{R/\fa}\longrightarrow R/\fc\longrightarrow
R/\fb\longrightarrow0$, where $K_{R/\fa}$ denotes the canonical
module of $R/\fa$.
  \item[(b)] If $R/\fa$ is a Cohen-Macaulay ring, then so is $R/\fb$.
\end{enumerate}
\end{lemma}

One of the classical and long-standing problems in linkage theory is
to determine what kind of certain properties of ideals which are
preserved under linkage. The main purpose of this paper is to give
answer to the following question.

\begin{question} \label{q}
Suppose that $(R,\fm)$ is a Gorenstein local ring, and $\fa$ and
$\fb$ are two linked ideals of $R$. When do the equalities $\depth
R/\fa=\depth R/\fb$ and $\cd(\fa,R)=\cd(\fb,R)$ hold?
\end{question}

Schenzel \cite{Sch2} has studied the formal local cohomology modules
$\ \mathfrak{F}^i_{\fa}(M):={\vpl}_nH^i_{\fm}(M/\fa^n M)$ for a
finitely generated module $M$ over a local ring $(R,\fm)$ and $i \in
\mathbb {Z}$. It is called $i$th formal local cohomology of $M$ with
respect to $\fa$. In the case of a Gorenstein local ring
\begin{equation}\label{formal}
\mathfrak{F}^i_{\fa}(R)\cong\Hom_R(H^{\dim R-i}_{\fa}(R),E),
\end{equation}
 where $E$ is the injective hull of $R/\fm$ (cf. \cite[Remark
3.6]{Sch2}). For more information on these kind of modules and their
properties we refrer the reader to \cite{A-D} and \cite{E}.

The first non-vanishing value of $\mathfrak{F}^i_{\fa}( M)$ is
called the formal grade as follows
$$\fgrade(\fa,M) := \inf \{i \in \mathbb {Z} : \mathfrak{F}^i_{\fa}( M) \neq 0\}.$$
In case $(R,\fm)$ is a Cohen-Macaulay local ring, then the equality
\begin{equation}\label{asdiv}
\cd(\fa,R)= \dim R- \fgrade (\fa,R)
\end{equation}
holds, (cf. \cite[Corollary 4.2]{A-D}).

Recently formal local cohomology has been used as a technical tool
to solve some problems, see for instance \cite{E2}. In this paper we
use it to give information on the cohomological dimension of an
ideal.

The outline of the paper is as follows. In Section $2$, we consider
the cohomological dimension of an ideal and a partial reverse
statement of a result of Peskin-Szpiro. In Section $3$, we give some
answers to the Question \ref{q} sorted below.

{\bf For the equality of cohomological dimensions:}
  \begin{itemize}
  \item There is an example of a regular local ring possessing two
  Cohen-Macaulay linked ideals while their cohomological dimensions
  are not the same (cf. Example \ref{minor}).
  \item For a regular local ring of positive characteristic
  with Cohen-Macaulay linked ideals, we give an affirmative answer (cf. Theorem \ref{cd}(a)).
  \item Another positive answer happens for evenly linked squarefree monomial
  ideals in the ring $k[x_1,...,x_n]_{(x_1,...,x_n)}$ (cf. Theorem \ref{cd}(c)).
\end{itemize}

{\bf For the equality of depths:}
  \begin{itemize}
  \item In a Gorenstein local ring, the depths of residue class rings
  of evenly linked ideals are the same(cf. Corollary \ref{depthelink}).
  \item We give negative answers in Corollary \ref{1} and Corollary
  \ref{2}.
  \item In a Gorenstein local ring, every 2-dimensional
  non Cohen-Macaulay linked ideals provide an affirmative answer (cf. Corollary \ref{3}).
\end{itemize}

\section{Some remarks on cohomological dimension}

Throughout this section, we assume that $(R,\fm)$ is a local ring
and $\fa$ is an ideal of $R$.  The first non-vanishing cohomological
degree of the local cohomology modules $H^i_{\fa}(R)$ is well
understood. It is the common length of maximal $R$-regular sequences
in $\fa$. The last non-vanishing amount of local cohomology modules,
instead, is more mysterious. In the light of Grothendieck's
Vanishing Theorem, $\cd(\fa,R) \leq \dim R$ and the equality happens
when $\fa$ is
 an $\fm$-primary ideal. In the following we characterize the last
 non-vanishing amount of local cohomology modules.

\begin{proposition} \label{characterize}
Let $\fa$ be an ideal of a  local ring $R$ and $N$ be a finite
$R$-module. Let $t$ be a positive integer. Then the following are
equivalent:
\begin{enumerate}
\item[(a)] $ H^{i}_{\fa}(N)=0$ for all $i > t$,
\item[(b)] $ H^{i}_{\fa}(N)$ is finite for all $i > t$,
\item[(c)] $\fa \subseteq \Rad(0:_R H^i_{\fa}(N))$ for all $i > t$.
\end{enumerate}
\end{proposition}
\begin{proof}
The implications $(a) \Longrightarrow (b)$ and $(b) \Longrightarrow
(c)$ are clear. In order to prove $(c) \Longrightarrow (a)$, we
argue by induction on $d:= \dim N$. For $d=0$, the claim is clear.
Now assume that $d>0$ and $\fa \subseteq \Rad(0:_R H^i_{\fa}(N))$
for all $i > t$. We are going to show that $ H^{i}_{\fa}(N)=0$ for
all $i > t$. First note that $H^{i}_{\fa}(N/H^{0}_{\fa}(N)) \cong
H^{i}_{\fa}(N)$ for all $i>t(>0)$ (cf. \cite{Br-Sh}). So, we may
assume that $H^{0}_{\fa}(N)=0$. Then there exists an $N$-regular
element $r$ in $\fa$. Consider  the short exact sequence
$$0 \rightarrow N \stackrel{r}{\rightarrow} N \rightarrow N/rN \rightarrow 0.$$
which implies the following long exact sequence
\begin{equation}\label{les}
\cdots\rightarrow H^{i}_{\fa}(N) \stackrel{r}{\rightarrow}
H^{i}_{\fa}(N) \rightarrow H^{i}_{\fa}(N/rN) \rightarrow
H^{i+1}_{\fa}(N)
 \rightarrow\cdots
\end{equation}

 for all $i$. In the light of \ref{les} and the induction
 assumption, one can deduce that $H^{i}_{\fa}(N/rN)=0$ for all $i>t$
 since $\dim N/rN=d-1$. Again using \ref{les}, we have $H^{i}_{\fa}(N)=r
H^{i}_{\fa}(N)$ for all $i>t$. As $\fa \subseteq \Rad (0:_R
H^{i}_{\fa}(N))$, one has $r^u H^{i}_{\fa}(N)=0$ for some integer
$u$, and so $ H^{i}_{\fa}(N)=0$ for all $i > t$.
\end{proof}

It should be noted that it has been shown in \cite[Proposition
3.1]{Y} that statements (a) and (b) in Proposition
\ref{characterize} are equivalent.

In the light of Proposition \ref{characterize}, we have
\begin{eqnarray*}
  \cd(\fa,R)\!\!\!\! &=&\!\!\!\! \sup\{i\mid H^{i}_{\fa}(R)\  \text{ is not finite}\}  \\
            \!\!\!\! &=&\!\!\!\! \sup\{i\mid \fa \not\subseteq \Rad(0:_R
             H^i_{\fa}(R))\}.
\end{eqnarray*}

In their valuable paper, Peskine-Szpiro (cf. \cite[Chap. IV]{P-S})
have proved the following result.

\begin{remark} \label{3.3}
Suppose $(R,\fm)$ is a regular local ring of positive characteristic
and $t$ is a positive integer. If $\depth R/\fa \geq t$, then
$\cd(\fa,R) \leq \dim R-t$.
\end{remark}

\begin{proof} The claim is clear by virtue
of \cite[Remark 3.1]{E2} and the equality (\ref{asdiv}).
\end{proof}

Next, we prove a slightly reverse statement of the above mentioned
result of Peskin-Szpiro.

\begin{lemma} \label{rev}
Let $R=k[x_1,\ldots,x_n]$ be a polynomial ring in $n$ variables over
a field $k$ with maximal ideal $\fm:=(x_1,\ldots,x_n)$ and $\fa$ be
a squarefree monomial ideal of $R$. Assume that $t$ is a
non-negative integer. If $\cd(\fa,R) \leq n-t$, then  $\depth R/\fa
\geq t$.
\end{lemma}
\begin{proof}
Since each of the modules is graded, so the issue of vanishing is
unchanged under localization at $\fm$. Suppose that $\cd(\fa,R) \leq
n-t$. This means $H^i_{\fa}(R)=0$ for all $i>n-t$. Hence, by
(\ref{formal}) we have $\mathfrak{F}^{j}_{\fa}(R)=0$ for all $j<t$.
So by virtue of \cite[Corollary 4.2]{E2}, we have $\depth
R/\fa=\fgrade(\fa,R)\geq t$.
\end{proof}

\begin{corollary} \label{correv}
Suppose that $R=k[x_1, \ldots, x_n]_{(x_1, \ldots, x_n)}$
  is a polynomial ring over a field $k$ of positive
  characteristic localized at $(x_1, \ldots, x_n)$ with $\fa$ is a
  squarefree monomial ideal. Then
\begin{enumerate}
   \item[(a)] If $\cd(\fa,R) \leq \cd(\fb,R)$, then
   $\depth R/\fa \geq \depth R/\fb$ for any ideal $\fb$ of
  $R$.
    \item[(b)] If $\depth R/\fa \geq \depth R/\fb$,
    then $\cd(\fa,R) \leq \cd(\fb,R)$ for any squarefree monomial ideal
   $\fb$ of $R$.
   \item[(c)] $\cd(\fa,R)=\cd(\fb,R)$
  if and only if $\depth R/\fa=\depth R/\fb$ for any squarefree monomial ideal
   $\fb$ of $R$.
\end{enumerate}
\end{corollary}
\begin{proof}
Assume that $\cd(\fa,R) \leq \cd(\fb,R)$. From Lemma \ref{rev}, the
equality (\ref{asdiv}) and \cite[Remark 3.1]{E2} we have $\depth
R/\fa \geq n-\cd(\fb,R)=\fgrade(\fb,R)\geq \depth R/\fb$. This
proves (a). In order to prove (b), assume that $\depth R/\fa \geq
\depth R/\fb$. To the contrary, suppose that $\cd(\fa,R) >
\cd(\fb,R)$. It implies that $n-\depth R/\fa > \cd(\fb,R)$. This in
conjunction with Lemma \ref{rev} yields $\depth R/\fb > \depth
R/\fa$ which is a contradiction. Now, (c) follows from (a) and (b).
\end{proof}

\section{Stability of
cohomological dimension under linkage}

Let $\fa$ and $\fb$ be ideals of a local ring $(R,\fm)$. In this
section, we try to find out some conditions for the stability of
cohomological dimension under linkage, i.e. $\cd(\fa,R)=\cd(\fb,R)$
whenever $\fa$ is linked to $\fb$. It should be noted that the
cohomological dimension is not preserved under linkage, in general.
To be more precise, consider the next example.

\begin{example}\label{minor}
Let $X=(x_{ij})$ be the generic $4\times3$ matrix and $X'$ be the
result of dropping the first two rows from $X$. Let $R=k[X]_{(X)}$
be a polynomial ring  over a field $k$ localized at the maximal
ideal $(X)$. Put $\fa:=I_3(X)$ (the ideal generated by the 3-minors
of $X$ in $R$) and $\fb:=I_2(X')$ (the ideal generated by the
2-minors of $X'$ in $R$). By virtue of \cite[Proposition 21.24]{Ei}
we see that $\fa$ and $\fb$ are Cohen-Macaulay linked ideals, but
$\cd(\fa,R)=4$ while $\cd(\fb,R)=3$ (cf. \cite[Corollary, pp.~
440]{Bru-Schw}).
\end{example}

According to Corollary \ref{correv} and Corollary 4.1 of \cite{E2},
under certain conditions, we see some relations between
$\cd(\fa,R)$, $\depth R/\fa$ and $\fgrade(\fa,R)$.

One of the technical tools we use to show the stability of
cohomological dimension under linkage is the depth, i.e. we
investigate the equality $\depth R/\fa=\depth R/\fb$. Take into
account that from the stability of depth one can not deduce the
stability of cohomological dimension under linkage, in general, as
the following concrete example shows it.

\begin{example} \label{concrete}
Let $R=k[x_0,x_1,x_2,x_3]$ be a polynomial ring over a field $k$.
Let $\fa=(x_0,x_1)\cap(x_2,x_3)$ be the defining ideal of the union
of the two skew lines in ${\mathbb P}^3$ and
$\fb=(x_0x_3-x_1x_2,x^3_1-x^2_0x_3,x^3_2-x_1x^2_3)$ be the defining
ideal of the  twisted quartic curve in ${\mathbb P}^3$. Then it is
not hard to show that
$$\fa \cap \fb=(x_0x_3-x_1x_2,x_0x_2^2-x_1^2x_3)$$
is a complete intersection. Therefore $\fb$ is linked to $\fa$ by
$\fa \cap \fb$. Using CoCoA \cite{cocoa}, we see that
$\depth(R/\fa)=1=\depth(R/\fb)$. But $\cd(\fa,R)=3$ and
$\cd(\fb,R)=2$.
\end{example}

According to the above examples and what we will see in the sequel,
the cohomological dimension, formal grade and depth are not stable,
in general, under linkage.

\begin{proposition} \label{depthlink}
Suppose that $(R,\fm)$ is a Gorenstein local ring, and $\fa$ and
$\fb$ are two linked ideals of R with $t:=\dim(R/\fa)$. Then
$\depth(R/\fb)=\depth(K_{R/\fa})-1$ provided that $R/\fa$ is not a
Cohen-Macaulay ring.
\end{proposition}
\begin{proof}
By the assumption, we have the following exact sequence
$$0\longrightarrow K_{R/\fa}\longrightarrow R/\fc\longrightarrow
R/\fb\longrightarrow0.$$ The above exact sequence together with
\cite[Proposition 1.2.9]{Bru-Her} yield the following inequalities
\begin{eqnarray}
\label{dwa}\depth(K_{R/\fa}) &\geq& \min\{\depth(R/\fc), \depth(R/\fb)+1\}, \\
\label{db}\depth(R/\fb) &\geq& \min\{\depth(K_{R/\fa})-1,
\depth(R/\fc)\}.
\end{eqnarray}
Note that, by Lemma \ref{Schenzel}(b), $R/\fb$ is not
Cohen-Macaulay. So using \ref{dwa}, one can deduce that
$\depth(K_{R/\fa})\geq\min\{t, \depth(R/\fb)+1\}=\depth(R/\fb)+1$.
On the other hand, by \ref{db}, we have
$\depth(R/\fb)\geq\min\{\depth(K_{R/\fa})-1,
t\}=\depth(K_{R/\fa})-1$. Therefore
$\depth(R/\fb)=\depth(K_{R/\fa})-1$.
\end{proof}

\begin{corollary} \label{depthelink}
Suppose that $(R,\fm)$ is a Gorenstein local ring, and $\fa$ and
$\fb$ are evenly linked ideals of R. Then
$\depth(R/\fa)=\depth(R/\fb)$.
\end{corollary}

There exists a Buchsbaum quasi-Gorenstein domain $A$ which is a
homomorphic image of a polynomial ring over a field $k$ with 11
indeterminates such that $\dim A=3$ and $\depth A=2$ (cf.
\cite{H-T}). Recall that a local ring $(A,\fn)$ is quasi-Gorenstein
if $H^{\dim A}_{\fn}(A)\cong E_A(A/\fn)$ the injective hull of
$A/\fn$. Now, we have a prime ideal $\fp$ in $R$ (a polynomial ring
over a field $k$ with 11 indeterminates) with $\dim(R/\fp)=3$ and
$\depth(R/\fp)=2$ such that $K_{R/\fp}\cong R/\fp$. Put
$\fb:=\fx:\fp$, where $\fx$ is a maximal regular sequence contained
in $\fp$. Then $\fb$ is linked to $\fp$. But
$$\depth(R/\fb)=\depth(K_{R/\fp})-1=1\neq
2=\depth(R/\fp),$$ where the first equality follows by Proposition
\ref{depthlink}.

This example encouraged us to bring the following corollaries.

\begin{corollary}\label{1}
Suppose that $\fa$ is an unmixed, non Cohen-Macaulay and
quasi-Gorenstein ideal in the Gorenstein ring $(R,\fm)$. Further
assume $\fb$ is an ideal linked to $\fa$. Then
$\depth(R/\fb)=\depth(R/\fa)-1$.
\end{corollary}

\begin{corollary}\label{2}
Suppose that $\fa$ is an unmixed, non Cohen-Macaulay and Buchsbaum
ideal in the Gorenstein ring $(R,\fm)$. Suppose that $K_{R/\fa}$ is
Cohen-Macaulay and $\dim(R/\fa)>2$. Further assume that $\fb$ is an
ideal linked to $\fa$. Then $\depth(R/\fb)>1=\depth(R/\fa)$.
\end{corollary}
\begin{proof}
First, note that by \cite[Theorem 5.4]{G-Sh} we have
$H_{\fm}^i(R/\fa)=0$ for all $1< i < \dim(R/\fa)$. Then
$\depth(R/\fa)=1$. On the other hand, one has
$$\depth(R/\fb)=\depth(K_{R/\fa})-1=\dim(K_{R/\fa})-1>1.$$
This completes the proof.
\end{proof}

As a consequence of Proposition \ref{depthlink} we derive the
following result for the equality of depth under linkage. It shows
that the equality of depth in Example \ref{concrete} is not
haphazard.

\begin{corollary}\label{3}
Suppose that $(R,\fm)$ is a Gorenstein local ring, and $\fa$ and
$\fb$ are two linked ideals of R with $\dim(R/\fa)=2$. Then
$\depth(R/\fb)=\depth(R/\fa)=1$ whenever $R/\fa$ is not
Cohen-Macaulay.
\end{corollary}
\begin{proof}
As $K_{R/\fa}=\Hom(R/\fa,R/\fx)$, where $\fx$ is a maximal regular
sequence in $\fa$, then by \cite[Exercise 1.4.19]{Bru-Her} we can
see that
\begin{eqnarray*}
  2&\geq &\depth(K_{R/\fa})\\
   &\geq & \min\{2,\depth(R/\fx)\} \\
   &=& \min\{2,\dim(R/\fx)\}=2.
\end{eqnarray*}
So by Proposition \ref{depthlink} we have $\depth(R/\fb)=1$. With a
similar argument we get $\depth(R/\fa)=1$.
\end{proof}

\begin{theorem} \label{cd}
Suppose that $\fa$ and $\fb$ are ideals of a local ring $(R,\fm)$.
\begin{enumerate}
  \item[(a)] Further assume $(R,\fm)$ is a regular local ring of
  positive characteristic and $\fa$ is linked to $\fb$.
  If $R/\fa$ is a Cohen-Macaulay ring, then $\cd(\fa,R)=\cd(\fb,R)$.
  \item[(b)] Assume that $R=k[x_1, \ldots, x_n]_{(x_1, \ldots, x_n)}$
  is a polynomial ring over a field $k$ of positive
  characteristic localized at $(x_1, \ldots, x_n)$ with $\fa$ is a
  squarefree monomial ideal.
  If $\fa$ is evenly linked to $\fb$, then
  $\cd(\fb,R)\leq\cd(\fa,R)$.
   \item[(c)] Assume that $R=k[x_1,...,x_n]_{(x_1,...,x_n)}$ is a polynomial ring over a
field $k$ localized at $(x_1, \ldots, x_n)$. If $\fa$ and $\fb$ are
two evenly linked squarefree monomial ideals, then one has
$\cd(\fa,R)=\cd(\fb,R)$.
\end{enumerate}
\end{theorem}
\begin{proof}
(a) Our assumptions in conjunction with \cite[3.1]{E2} prove the assertion.\\
(b) As $\fa$ is a square free monomial ideal, then  \cite[4.2]{E2}
implies that $\depth R/\fa= \fgrade(\fa,R)$. Since  $\fa$ is evenly
linked to $\fb$ then $\depth R/\fa=\depth R/\fb$ (cf. Corollary
\ref{depthelink}). Now,
the claim follows by \cite[3.1]{E2} and the equality (\ref{asdiv}).\\
(c) The claim follows by Corollary \ref{depthelink} and
\cite[Theorem 4.2]{E2}.
\end{proof}

It is noteworthy to say that assumptions in Theorem \ref{cd}(a) are
not too much as we have seen in Example \ref{minor}.

\begin{remark}
Suppose that $\fa$ is an ideal in a regular local ring $(R,\fm)$ of
positive characteristic in a linkage class of a complete
intersection ideal. Immediately, it follows from Theorem \ref{cd}(a)
that $\Ht\fa=\cd(\fa,R)$, that is $\fa$ is a cohomologically
complete intersection ideal.
\end{remark}

\subsection*{Acknowledgment}
The authors are deeply grateful to Professor K. Divaani-Aazar for
careful reading of the first draft of the paper and useful comments. 


\end{document}